\documentclass[12pt,leqno]{amsart}
\usepackage[dvips]{graphics}
\usepackage{amssymb,hyperref}
\usepackage{color,comment}

\numberwithin{equation}{section}

\newtheorem{thm}{Theorem}[section]

\newtheorem{lem}[thm]{Lemma}
\newtheorem{cor}[thm]{Corollary}
\newtheorem{prop}[thm]{Proposition}

\newtheorem{quest}[thm]{Problem}

\newtheorem{defn}[thm]{Definition}
\theoremstyle{remark}

\newtheorem{ex}[thm]{Example}

\newcommand{\tref}[1]{Theorem~\ref{#1}}
\newcommand{\cref}[1]{Korollar~\ref{#1}}

\newcommand{\R}{\mathbb{R}}

\begin{document}
\pagebreak


\title{Remarks on manifolds with two sided curvature bounds}

\thanks{V.K.  is partially supported by a Discovery grant from NSERC;
	A. L. was partially supported by the DFG grants   SFB TRR 191 and SPP 2026.}
\author{Vitali Kapovitch}
\address{University of Toronto}
\email{vkt@math.toronto.edu}

\author{Alexander Lytchak}

\address
{Mathematisches Institut\\ Universit\"at K\"oln\\ Weyertal 86 -- 90\\ 50931 K\"oln, Germany}
\email{alytchak@math.uni-koeln.de}


\date{\today}


\begin{abstract}
	We discuss  folklore statements about distance functions in manifolds
	with two sided bounded curvature.  The topics include regularity, subsets of positive reach and the cut locus.
\end{abstract}



\maketitle

\section{Introduction}
\subsection{Distance functions in smooth Riemannian manifolds}
We discuss slightly generalized versions of some 
folklore results about distance functions $d_A$ to subsets $A$ of smooth Riemannian manifolds $M$.
The results turn out to be local and  are proved without any completeness  assumptions.  
Our  proofs  do not involve Jacobi fields, but only basic facts about semi-concavity and semi-convexity  of distance functions.
Therefore, the statements generalize to the synthetic setting of manifolds with two sided curvature bounds, as will be explained below.

The first statement is   well-known in the complete situation, see, for instance,  \cite{Mante}, \cite{Sine}, \cite[Section 2]{Ghomi}.

\begin{prop} \label{prop: distance}
Let  $M$ be a smooth Riemannian  manifold. Let $A\subset M$ be a closed subset and let $f$
denote the distance function to the set $A$.  Then $f$ is semiconcave in $M\setminus A$.  The following conditions are equivalent for an open subset $O\subset M \setminus A $:
\begin{enumerate}
\item $f$ is semiconvex in $O$.
\item $f$ is $\mathcal C^{1,1}$ in $O$.
\item  $f$ is $\mathcal C^1$ in $O$.

\item  For any $y\in O$, there exists at most one geodesic $\gamma_y :[0,\epsilon)\to M$ starting at $y$ and parametrized by arclength, with the property 
  $f\circ \gamma_y (t) =f(y)-t$, for all $t$ in $[0,\epsilon)$.
\end{enumerate}
\end{prop}

We assume some familiarity with the notions of semiconcave (semiconvex) functions and their gradient flows, \cite{Petrunin-semi}, \cite{AKP}, \cite{Gigli}.
Here and below,  the notions of semi-concavity (convexity) and $\mathcal C^{1,1}$ are local. For instance, we say  that a function $f$ is  $\mathcal C^{1,1}$ if it is $\mathcal C^1$ and the gradient $\nabla f$ is \emph{locally}  Lipschitz continuous.  

For a closed subset $A$ of $M$, denote by $Reg ^{d_A}$
the set of points $x\in M\setminus A$, such that $d_A$ is $\mathcal C^1$ in a neighborhood of $x$.  Thus, $Reg ^{d_A}$ is the maximal open set $O\subset M\setminus A$ which satisfies the equivalent conditions  of Proposition \ref{prop: distance}. The closed subset $CL(A)= (M\setminus A) \setminus Reg ^{d_A}$ of $M\setminus A$ is called the   \emph{cut locus} of $A$ in $M$.

If $M$ is  complete, then the geodesic $\gamma _y$  in (4) can be extended to a unique minimizing geodesic  $\gamma _y:[0, d_A(y)]  \to M$ from $y$ to $A$.   The endpoint of this geodesic is  the unique projection point $\Pi ^A (y)$ of $y$ on $A$.  Moreover, the whole geodesic $\gamma_y ([0, d_A(y)))$ is contained in $Reg ^{d_A}$.   

If $M$ is complete and $A$ is a $\mathcal C^k$ submanifold, for $k \geq 2$, then  $d_A: Reg ^{d_A} \to \R$ can be expressed in terms of the normal exponential map of $A$   and turns out to be  $\mathcal C^k$, see, for instance, \cite[Prop. 4.3]{Mante}. 

In the non-complete situation, none of the above statements need to hold, see Example \ref{ex: nonsmooth}.

The next result  is   known to specialists   for some subsets  in  complete manifolds,  \cite{Alb1}, \cite{Alb2}. 
  We inlcude a  short proof  based on a general fact about gradient flows. Note, that the function $d_A$ is  semi-concave  on $M\setminus A$ and has a uniquely defined (local) gradient flow.


\begin{prop} \label{prop: cut}
	Let $A$ be a closed subset of a Riemannian manifold $M$.  Then 
	the cut locus $CL(A)\subset M\setminus A$  of $A$ is invariant under the gradient flow $\Phi $ of the distance function  $d_A$. 
\end{prop}
 
 A natural generalization of the above result and its proof is valid in Alexandrov spaces, \cite[Prop. 14.1.5]{AKP}. As an application of Proposition \ref{prop: cut} one can derive a very short  proof of the nice geometric observation \cite[Theorem 6.1]{Ghomi}, see Corollary \ref{cor: ghomi} below.

 Finally, we address  (again essentially well-known to specialists) properties of subsets of positive reach, defined and investigated  by Federer  in Euclidean spaces, \cite{Federer},  and
 by Bangert and Kleinjohann in Riemannian manifolds, \cite{Bangert}, \cite{Kleinjohann}; see also \cite{Ly-conv}, \cite{Ly-reach}, \cite{Rataj}.
 
   Recall that  a closed subset $A$ in a Riemannian manifold  is said to have positive reach, if $A$ has a neighborhood $O$ such that the foot point projection $\Pi^A$ is uniquely defined on $O$.
 As has been shown by Federer and Bangert, the notion only depends on the underlying smooth structure and not on the Riemannian metric. The following result is  essentially contained in \cite{Kleinjohann}, \cite{Bangert}, \cite{Ly-conv}.
 
 
 \begin{prop} \label{prop: reach}
 	Let $A$ be a closed subset of a smooth Riemannian manifold  $M$.
 	Then the following  are equivalent:
 	\begin{enumerate}
 		\item The subset $A$ is of positive reach.
 		\item There is an open neighborhood $O$ of $A$ such that the distance function $d_A$ 
 		is $\mathcal C^{1,1}$ in $O\setminus A$.
 		\item The  function $d_A$ is semiconvex on a  neighborhood $O$ of $A$.
 	\end{enumerate} 
 \end{prop} 
 
 See \cite{Bangert} and \cite{Ly-conv} for other characterizations.
 
 
Any subset of positive reach $A\subset M$ has a well-defined \emph{tangent cone} $T_xA$ at every point $x\in A$.  This tangent cone is a convex cone  in $T_xM$, \cite[Theorem 4.8]{Federer}. The \emph{normal cone}
$T ^{\perp}_x A$ is the convex cone of all vectors in $T_xM$ enclosing angles at least $\frac \pi 2$ with all vectors in $T_xA$.

For the following  folklore statement about subsets of positive reach we could not find appropriate references. In the Euclidean case the result is contained in \cite{Federer}.

\begin{prop}\label{lem-norm-geod}
		Let $A\subset M$ be a subset of positive reach. Then for any  $x\in A$ and unit $h\in T^{\perp}_xA$ the geodesic $\gamma ^h$ starting at $x$ in the direction of $h$
		satisfies $d_A (\gamma ^h (s)) = s$, for all $s>0$ such that $\gamma ^h ([0,s])$  is contained in the open subset $O$ from Proposition \ref{prop: reach}.
	\end{prop}

\subsection{Manifolds with two sided curvature bounds} \label{subsec: intro}
Here we discuss some basic observations about non-smooth Riemannian manifolds which have two sided curvature bounds in the sense of Alexandrov and 
the extensions of the above results to this setting. These results have been applied in \cite{KL}. Readers only interested in the smooth situation can skip this part of the introduction together with Section \ref{sec: last}.

Manifolds with two sided curvature bounds in the sense of Alexandrov appear in one of  the most prominent examples 
of Gromov--Hausdorff convergence. Namely, the class of such compact manifolds with uniform bounds on injectivity radius, diameter and curvature is compact with respect to the Gromov--Haudorff convergence, \cite{Ber-Nik}, \cite{Gromov}, \cite{Peters}.
It provides   a natural compactification  of the corresponding class of smooth Riemannian manifolds.  Moreover, all manifolds with two sided curvature bounds in the sense of Alexandrov turn out to  have a rather regular analytic structure and to admit a smoothing, as has been proved by Nikolaev in a series of papers.
 A  good readable  summary of these results of Nikolaev and related statements on the structure of such manifolds 
  has appeared in \cite{Ber-Nik}.  
 Results and ideas of Nikolaev motivated many  theorems in the theory of Alexandrov spaces with one-sided curvature bounds,  \cite{BGP}, \cite{Kap-Kell-Ket-19},  \cite{Per-DC}, \cite{OS},  \cite{Petruninpar}, \cite{Lytchak-Nagano18}. 
 

    We will assume some familiarity with the theory of Alexandrov spaces, and refer the non-familiar reader to   \cite{AKP}.  
    The most appropriate setting for our  local results is the following one. 
    
    \begin{defn} \label{def}
    A locally compact length metric space $X$ has two sided bounded curvature if for any point $x\in X$ there exists a compact convex neighborhood $U$ and some $K>0$ such that $U$ is an Alexandrov space of curvature $\geq -K$ and a $CAT(K)$ space.
    \end{defn}


 A space $X$ with two sided bounded curvature is topologically a manifold  $M$ with boundary $\partial M$.  Moreover, $M\setminus \partial M$  is convex in $M$, \cite{Ber-Nik}, cf. 
 \cite{Kap-Kell-Ket-19}. We restrict  the attention to the case $\partial M =\emptyset$.  
 
 By a \emph{manifold with two sided curvature bounds} we will denote a space as in the above Definition   \ref{def} which, in addition,  is homeomorphic to a manifold without boundary.

 Any manifold $M$ with two sided curvature bounds admits a natural atlas of distance coordinates,
 see \cite{Ber-Nik} and Section \ref{sec: dist} below. 
     The distance in $M$ is defined by a $\mathcal C^{0,1}$
 Riemannian metric $g$ in  this atlas. Moreover, the $\mathcal C^{1,1}$-smoothness of the atlas 
 and the $\mathcal C^{0,1}$-smoothness of the Riemannian metric is optimal, as can be observed in the manifold $M$ arising from the gluing of a flat cylinder and a hemisphere.
  Results of this type with weaker conclusions  have been obtained  for  distance coordinates under one-sided curvature bounds, \cite{Per-DC}, \cite{OS}, \cite{Lytchak-Nagano18}.

 On any $\mathcal C^{1,1}$ manifold $M$  with a  Lipschitz continuous Riemannian metric, any harmonic function is  $\mathcal C^{1,\alpha}$, for all $\alpha <1$, \cite[III, Chapter 9]{Taylor} and 
 there exist harmonic coordinates around any point, \cite[III, Chapter 9]{Taylor}.  The atlas of harmonic coordinates is of class $\mathcal C^{2,\alpha}$, \cite[p. 689]{Sabitov}
and the distance is defined by a metric of class $\mathcal C^{\alpha}$ in these coordinates.  One of the  central results of Nikolaev's theory, see \cite{Ber-Nik}, is that for any 
 manifold $M$ with two sided curvature bounds,  the harmonic atlas 
 is of class $\mathcal C^{3,\alpha}$, for any $\alpha <1$, and the Riemannian metric in this atlas is of class  $\mathcal C^{1,\alpha}$, for any $\alpha<1$.

 In a general $\mathcal C^{1,1}$-manifold with a $\mathcal C^{0,1}$ Riemannian metric a harmonic function does not need to be of class $\mathcal C^{1,1}$, cf. \cite[p. 693]{Sabitov} and Problem \ref{problem}  below.  Our first observation is that such a  loss of smoothness cannot happen on a manifold with two sided bounded curvature.

 \begin{prop} \label{prop: harm}
 	Let $M$ be a manifold with two sided bounded curvature. Then any harmonic function 
 	on an open subset $U$ in $M$ is of class $\mathcal C^{1,1}$ in distance coordinates. Thus,  any transformation from distance to harmonic coordinates is of class 
 	$\mathcal C^{1,1}$.
 \end{prop}

 Thus, the $\mathcal C^{1,1}$ atlas of distance coordinates can be assumed to include all harmonic coordinates as well. From now on all statements will refer to this $\mathcal C^{1,1}$ atlas.
 The following result might  be folklore knowledge:
 \begin{prop} \label{prop: submanifold}
 	Let $M$ be a manifold with two sided curvature bounds and   let $N\subset M$ be a 
 	$\mathcal C^{1,1}$ submanifold.   Then $N$ with its intrinsic metric 
 	is a manifold with two sided curvature bounds.
 \end{prop}
 
   This result follows from the Gau\ss $ $ equation and another central result of Nikolaev's theory, stating that manifolds with two sided curvature bounds are exactly the limit spaces (in a precise local sense)  of smooth Riemannian manifolds with uniform bounds on sectional curvature.

 The final statement discussed in the introduction is that all results about distance functions in smooth Riemannian manifolds are valid in this more general setting:
 
 \begin{thm}
 The statements of Proposition \ref{prop: distance}, Proposition \ref{prop: cut} and Proposition \ref{prop: reach} are valid for any closed subset $A$ of any manifold $M$ with two sided bounded curvature.  The differentiability 
 in Propositions \ref{prop: distance}, \ref{prop: reach} is considered with respect to distance coordinates.
 	\end{thm}

 \subsection{A few questions}
 



We would like to finish the introduction with a few open questions about manifolds with two sided curvature bounds.

\begin{quest} \label{problem}
Does there exist a Riemannian metric of class $\mathcal C^{0,1}$ and a harmonic function with respect to this metric, which is not $\mathcal C^{1,1}$?	
\end{quest} 

It is possible that a positive answer to  this problem  might  be obtained following the ideas in the examples discussed in   \cite[p. 693]{Sabitov}.

The second part of the next problem  is motivated by Proposition \ref{prop: submanifold}. It should be compared with 
Nash's embedding theorem for Riemannian manifolds of higher regularity, see \cite{Andrews}.

\begin{quest}
Let $M$ be a (compact, complete, local) manifold with  two sided curvature bounds in the sense of Alexandrov.  Do there
exist coordinates on $M$ in which the distance is defined by a Riemannian metric of class $\mathcal C^{1,1}$?
Can $M$ be length-preserving embedded as a $\mathcal C^{1,1}$ submanifold in a Euclidean  space? 
\end{quest}

 We would like to mention that the folklore argument providing a negative answer to the first question  above,  for instance, \cite{Peters}, is not correct.
Indeed, \cite{Peters}
provides a surface with two sided curvature bounds such that in harmonic coordinates
  the Riemannian metric is not of class $\mathcal C^{1,1}$. Then the proof invokes the statement of \cite{KZ}, \cite{Sabitov} that the smoothness of the metric is optimal in harmonic coordinates.  However, this elliptic regularity statement is not covered by \cite{KZ} (since elliptic regularity does not work well for Lipschitz functions). Indeed, the following result has appeared in \cite[Example 2]{Sabitov} and provides a  counter-example to the folklore proof.  There exists a Riemannian metric of class $\mathcal C^2$ on $\R^2$ 
  such that
	in any harmonic coordinates the metric is not $\mathcal C^{1,1}$.

The second question concerns the existence of a canonical smoothing of  manifolds with two sided curvature bounds.
  While the existence part in the next question is a direct consequence of  the approximation results of Nikolaev and \cite{Shi}, the uniqueness is  more subtle:

\begin{quest}
 Let $M$ be a complete manifold with two sided bounded curvature. If the curvature bounds can be chosen uniformly on $M$ then  there exists a unique  Ricci flow coming out  of the manifold $M$.
\end{quest}

The final problem we would like to mention is due to the following fact. The theory of Nikolaev is scattered through several works and some of them are not easy to read (and to find).  Thus, we formulate:
\begin{quest}
	Find a streamlined proof of Nikolaev's result on smoothings of manifolds with two sided curvature bounds, \cite{Nik-closure}, \cite{Nik-synge}.  	
\end{quest}


\subsection{Structure of the paper}
In Section \ref{sec: sing}, we recall basics on semiconcave functions and their gradient flows and  verify a local  version of Proposition \ref{prop: cut}.  In Section \ref{sec: dist}, we recall
basic facts about distance coordinates. In Section \ref{sec: main}, we prove Proposition \ref{prop: distance} and \ref{prop: cut}.  In Section \ref{sec: reach}, we prove the stated results about subsets of positive reach, Propositions \ref{prop: reach}, \ref{lem-norm-geod}. 
Finally, in Section \ref{sec: last}, we prove Proposition \ref{prop: harm} and Proposition \ref{prop: submanifold}.

\section{Semiconcavity and gradient flows} \label{sec: sing}
\subsection{Notation}
Distance will be denote by $d$. The distance function from a subset $A$ of a space $X$ will be denote by $d_A$. By definition, this is a $1$-Lipschitz function. A geodesic  will denote an isometric embedding $\gamma:I\to X$ of an interval. Thus, our geodesics are always parametrized by arclength and globally minimizing.

\subsection{Special neighborhood} \label{subsec: spec}
Let $M$ be a manifold with two sided curvature bounds. For any point $x\in M$, we find a compact neighborhood $U$ of $x$ as in Definition \ref{def}.  Restricting the neighborhood and using convexity of small balls in $CAT(K)$ spaces we  may assume that the neighborhood $U=U_x$ has the following form.

The set $U_x$ is the closed ball of radius $r_x$ around $x$ and it is homeomorphic to a Euclidean ball, \cite[Theorem 12.1]{Ber-Nik}.   Any pair of points in $U_x$ is connected by a unique geodesic in $U_x$.
For some $K_x>0$, the space $U_x$ is $CAT(K_x)$ and an Alexandrov space of curvature $\geq -K_x$, moreover, $\epsilon_x :=r_x\cdot K_x <<1$.
Any geodesic in $U_x$  extends to a geodesic starting and ending on the distance sphere  $\partial U_x$, \cite[Prop. 8.3]{Ber-Nik}. 

\subsection{Semiconvexity, semiconcavity and gradient flows}
A locally Lipschitz function $f$ on an open subset $O$ of $M$ is  \emph{$C$-concave}, respectively \emph{$C$-convex},
if for any geodesic $\gamma:I\to O$  the function $f\circ \gamma  (t) -\frac C 2 t^2 $
is concave, respectively convex, on $I$.

A function $f$ is $C$-concave if and only if for all pairs of points $p_1,p_2 \in M$ which are sufficiently close to each other and any midpoint $m$ between $p_1$ and $p_2$, we have
\begin{equation} \label{eq: semi}
f(m) \geq \frac 1 2 (f(p_1) +f(p_2)) - \frac C 8 \cdot d^2(p_1,p_2)\;.
\end{equation}

The function $f:O\to \R$ is  \emph{semiconcave} (semiconvex) if for any $x\in O$ there is some 
$C \in \R$, such that the restriction of $f$ to some neighborhood of $x$ in $O$ is $C$-concave
($C$-convex).

If $f:O\to \R$ is semiconcave and $h:O\to \R$ is continuous, we say that 
the function $f$ is $h$-concave, if  for any $x\in O$ and any $\epsilon >0$, there exists a neighborhood $O_{\epsilon}$ of $x$
such that  $f$ is $(h(x)+\epsilon)$-concave in $O_{\epsilon}$.  

This is equivalent to the requirement that, 
for any geodesic $\gamma:I\to O$, we have 
$(f\circ \gamma ) '' \leq h\circ \gamma$ on $I$ in the sense of distributions.

Since the notion of semiconcavity and of gradient curves and flows of semiconcave functions is
local, the whole theory of gradient flows in Alexandrov spaces, \cite{AKP}, \cite{Petrunin-semi}, applies to the present situation.

For any semiconcave function $f:O\to \R$ and every $x\in O$ there exists a unique vector $\nabla _xf \in T_xO$, the \emph{gradient} of $f$ at $x$. Moreover, there exists a unique  maximal curve
$\eta _x:[0,a) \to O$ with some  $a\in (0, \infty)$, the gradient curve of $f$, which starts in $x$ and satisfies 
$$\eta_x'(t)
=\nabla_{\eta_x(t)} f \; \text{and} \;  (f\circ \eta _x)' (t) = |\nabla _{\eta_x (t)} f| ^2 \;,$$ for all $t \in [0,a)$.   Furthermore, if $a<\infty$ then  $\eta_x([0,a))$ is not contained in a compactum in $O$.  Finally, the map
$(x,t )\to \Phi (t,x):= \eta _x (t)$ is a local flow  defined on a neighborhood of $O\times \{0 \} \subset O\times [0,\infty)$ is locally Lipschitz continuous, \cite{Petrunin-semi}.

A point $x\in O$ is \emph{critical} for $f$ if  $\nabla _x f =0$. In this case $\eta_x$ is the stationary curve $\eta _x(t)=x$.  On the other hand, if the curve $\eta _x$ does not contain critical points of $f$, it has a unique parametrization $\tilde \eta _x :[0,\tilde a)\to O$ by arclength.

The following is a not very well-known but fundamental observation. The first statement is exactly    \cite[Theorem 14.1.3]{AKP}; the second one follows from the first  by localization:

\begin{lem} \label{lem: restr}
	Let $f:O\to \R$ be $C$-concave.  Then for the arclength reparametrization $\tilde \eta _x$
	of any gradient curve $\eta _x$ of $f$ in $O$, the composition $f\circ \tilde \eta _x:[0,a)\to \R$ is $C$-concave.

If $f$ is $h$-concave for a continuous function $h:O\to \R$ then 
$f\circ  \tilde \eta_x$ is $h\circ \tilde \eta_x$-concave on $[0,a)$.
\end{lem}  




\subsection{Distance functions in special neighborhoods} \label{subsec: distspec}
Let $U\subset M$ be a special neighborhood of some point as in Subsection \ref{subsec: spec}.
Thus, $U$ is convex, compact,  $CAT(K)$ and it is an Alexandrov space of curvature $\geq -K$, for some $K>0$.

Then, for any point $x\in U$ the distance function $f=d_x$ is convex in $U$. Moreover, 
$f$ is also semiconcave on $U\setminus \{x\}$.
More precisely, on the set $O$ of points $y\in U$ with $d(x,y ) >\delta$, the function
$f$ is $C$-concave, for some $C=C(K, \delta)$.  

Since an infimum of $C$-concave functions is $C$-concave, for any subset $B\subset U$ the distance function $d_B$ is  $C(K,\delta )$ concave on the set of all points    $y\in U$ with $d_B(y)>\delta$.

The function $f=d_B$ is $1$-Lipschitz, thus $|\nabla _y f| \leq 1$ for all $y\in U\setminus B$.
By the first variation formula, $|\nabla _y f|=1$ if and only if $y$ is connected with $B$ by a unique shortest geodesic.

The conclusion of Proposition \ref{prop: cut}  will easily follow from the next
Lemma, cf. \cite[Prop. 14.1.5]{AKP}, \cite[Theorem 4.5]{Alb1}:

 \begin{lem} \label{lem: concl}
	Let $B\subset U$ be any closed subset in a special neighborhood $U$ as above. Let $y\in U\setminus B$ be arbitrary and let $ \eta _y :[0, a ) \to U$ be the  gradient curve  in $U$ of the distance function $f=d_B$ starting at $y$.  
	If $|\nabla _y f|<1$ then $|\nabla _z f|<1$, for all $z$ on the gradient curve $\eta _y$.
\end{lem}

\begin{proof}
Rescaling the space, we may assume that the lower curvature bound on $U$ equals $-1$.
Then, for the distance function $F=d_p$ to any  $p\in U$ the composition $\hat F = \cosh \circ F$ is $\hat F$-concave on $U$, \cite[Theorem 7.4.1]{AKP}. Hence, also for  the infimum $f=d_B$ of distance functions to points, the composition
$\hat f:=\cosh \circ f$ is $\hat f$-concave on $U$.

  The gradient curves of $f$ and $\hat f$ on $U\setminus B$ coincide up to parametrizations,
  \cite[Theorem 11.4.4]{AKP}.  Thus, the arclength reparametrization $\tilde \eta _y$ of $\eta_y$    is also the arclength parametrization of the gradient curve of
  $\hat f$. From  Lemma \ref{lem: restr}, we deduce that 
  $\hat f\circ \tilde \eta _y $ is $\hat f\circ \tilde \eta _y$-concave  on the  interval of definition $[0, b]$ of $\tilde \eta _y$ from $y$ to $z$.
  
  Define  $h:[0,b]\to \R$ as $h(t)=  f\circ \tilde \eta _y (t)$.  Then $h$ is an increasing,  $1$-Lipschitz, semiconcave function; we have $h'(0)<1$ and the composition $\hat h:=  \cosh \circ h$ is $\hat h$-concave. We only need  to verify $h'(b) <1$.

  Assume, by contrary that $h'(b)=1$.  Consider the linear function $h_0 (t)= t +(h(b)-b)$.
  Then 
  $$h_0(b)=h(b) \; ; \;  h_0'(b)= h'(b) \; ; \; (\cosh \circ h_0 ) '' =\cosh \circ h_0 \;.$$
 By comparison, cf. \cite[Theorem 4.5.3]{AKP}, we deduce 
 $$\cosh \circ h(t) \leq \cosh \circ h_0 (t) \,, $$ 
 for all $t\in [0,b]$. Hence $h(t)\leq h_0(t)$ for all $t$.  However, $h$ is $1$-Lipschitz.
 Thus, $h(t)=h_0(t)$ for all $t\in [0,b]$. This implis $h'(0)=1$ in contradiction to our assumption. This finishes the proof.
\end{proof}

\subsection{Distance functions to remote subsets} \label{subsec: remote}
Let  $M$ be a manifold with two sided curvature bounds  (possibly non-complete and without uniform bounds on curvature), as in Subsection \ref{subsec: intro}.  Let $A$ be a closed subset of $M$ and consider the distance function $f=d_A$ to $A$.

Consider an arbitrary point $x\in M\setminus A$ and a special neighborhood $U=U_x =\bar B_{r_x} (x)$ as in Subsection \ref{subsec: spec}.  Making $r_x$ smaller, we may assume that $U$ is disjoint from $A$.

    Let $\tilde U$ be a smaller ball 
of radius $r< \frac 1 3 \cdot r_x$ around $x$.   Set $s= f(x)$ and denote by $B$ the compact subset $B=f^{-1} (s-2r) \cap U$.   By compactness and convexity of $U$, for any $y\in \tilde U$, we find 
at least one footpoint $\hat y$ of $y$ on $B$, thus $d(y,\hat y)= d_B(y)$.   By the triangle inequality, 
$$f(y) \leq f(\hat y) +d(y, \hat y)  = (s-2r)+ d_B(y) \;.$$
On the other hand, any curve $\eta$ from $A$ to $y$ must contain points $z$ with $f(z)=s-2r$. 
If $z$ is not in $U$ then the length of $\eta$ is at least $s+r$.  Thus, any such $\eta$ 
with length less than $s+r$, contains points on $B$. Thus by the triangle inequality we deduce the equality 
$$d_A=f= (s-2r) +d_B$$
on $\tilde U$.   Since $d_B$  is semiconcave on $\tilde U$, so is $f$, moreover, the gradient curves of $f$ and $d_B$ coincide in $\tilde U$.

Now we arrive at the following results (also valid in all Alexandrov spaces and their localized version, Alexandrov regions, \cite{Leb-Nep}). 
\begin{cor} \label{cor: petrun}
	Let $A$ be a closed subset in a manifold $M$ with two sided curvature bounds. 
	Then, the distance function $f=d_A$ is semiconcave on $M\setminus A$. For any point 
	$x\in M\setminus A$ with $|\nabla_x  f|<1$ we have $|\nabla _yf| <1$  for all $y$ on the gradient curve $\eta _x$.
\end{cor}

\begin{proof}
The semiconcavity condition is local and has been verified above in the neighborhood of any $x\in M\setminus A$.  

In order to see the second statement, we consider the compact part $\eta _x :[0,b] \to M$ of the gradient curve $\eta _x$ between $x$ and $y$. Assuming the contrary, and using that $f\circ \tilde \eta _x$ is semiconcave, we find a smallest 
$t \in (0,b]$ such that $|\nabla _{\eta_x (t)} f|=1$.

Now we find  a small special neighborhood $U$  of $z=\eta _x (t)$, identify on a smaller neighborhood $\tilde U$ (up to an additive constant) $f$ with $d_B$ for a closed subset $B\subset U$ and derive a contradiction  to 
  Lemma \ref{lem: concl}.
	\end{proof}

\section{Coordinates} \label{sec: dist}


\subsection{Distance coordinates}
Let $U\subset M$ be a special neighborhood of a point $x$ in a manifold $M$ with two sided curvature bounds.

  For any $x\in U$ the function $d_x$ is convex in $U$. Moreover, 
  $d_x$ is $C(K,\delta)$ concave on the set of points $y$ in $U$ with
  $d(x,y)>\delta$.

  For any  $p$ in the interior of $U$, consider any  points $p_1,...,p_n$ in $U$ such that the starting  directions
  of the geodesics $pp_i$ are almost orthogonal at $p$, (see \cite[Theorem 13.2]{Ber-Nik}).   Then the map $F:U\to \R^n$ with coordinates 
  $f_i:=d_{p_i}$  is a biLipschitz map $F:O\to \tilde O$ from an open ball $O$ around $p$ onto an open subset of $\R^n$.   (This is even true for any Alexandrov space, \cite{BGP}).
  
  Any such restriction $F:O\to \tilde O$ is called a \emph{distance chart} on $M$. The distance charts define a $\mathcal C^{1,1}$ atlas on $M$ and the distance on $M$ is given by a Riemannian metric of class $\mathcal C^{0,1}$ with respect to this atlas,  \cite[Theorem 13.2]{Ber-Nik}.

  \subsection{Distance charts and  semiconcavity}
  The following observation can also be used to obtain a shorter alternative proof of 
  \cite[Theorem 13.2]{Ber-Nik}, using the observation that a homeomorphism  between open subsets of $\R^n$ is $\mathcal C^{1,1}$ if and only if it preserves the class of semiconcave functions.
  
  \begin{lem}
  	Let $F:O\to \tilde O\subset \R^n$ be a distance chart in a manifold with two sided curvature bounds.   Then a function $f:\tilde O\to \R$ is semiconcave if and only if $f\circ F$ is semiconcave on $O$. 
  \end{lem}

  \begin{proof}
  From the biLipschitz porperty of $F$, and semiconcavity and semiconvexity of the coordinates $f_i$, we deduce the following.
  For any geodesic $\gamma$ in $O$ connecting $q_1$ and $q_2$ and having $m$ as its midpoint
  the distance between $F(m)$ and the midpoint $\bar m =\frac 1 2 (F(q_1) +F(q_2))$  in $\R^n$ between $F(q_i)$, we have 
  $$d(F(m),\bar m) \leq C\cdot d^2(q_1,q_2)\;.$$ 
  Here the constant $C$ depends only on the biLipschitz constant of $F$ and the curvature bounds.
  
  Since $F$ is biLipschitz,  $F^{-1}$ sends midpoints in $\tilde O$ to "almost midpoints" in $O$ in the same sense as above.   
  Now the equivalence of semiconcavity of $f$ and $f\circ F$ follows  after applying \eqref{eq: semi}.
  	\end{proof}

 \section{Main results} \label{sec: main}
 \subsection{General distance functions}  We are going to prove the following slight generalization of Proposition \ref{prop: distance}
 
 \begin{prop} \label{prop: distance1}
 	Let  $M$ be a manifold with two sided curvature bounds. Let $A\subset M$ be a closed subset and  $f=d_A$.  Then the following conditions are equivalent for an open subset $O\subset M\setminus A$:
 	\begin{enumerate}
 		\item \label{o-semiconvex} $f$ is semiconvex in $O$.
 		\item \label{o-c1-1} $f$ is $\mathcal C^{1,1}$ in $O$.
 		\item \label{o-c1} $f$ is $\mathcal C^1$ in $O$.
 		\item For all $x\in O$ we have $|\nabla _x f|=1$
 		\item \label{o-unique-geod} For any $x\in O$, there exists at most one geodesic $\gamma_x :[0,\epsilon)\to M$ starting at $x$, with  
 		$f\circ \gamma_x (t) =f(x)-t$, for all $t$ in $[0,\epsilon)$.
 	\end{enumerate}
 \end{prop}
 
    \begin{proof}
    	All statements are local on $O$. We may fix $p\in O$ and  consider  a special neighborhood $U$
    	of $p$ in $M$. Furthermore, in Subsection \ref{subsec: remote}, we have found a closed subset $B$ in $U$ and a smaller neighborhood
    	$ O_0$ of $p$ in $O$, such that on $O_0$ the function $f$ coincides with 
    	$d_B$ up to an additive constant. Thus, we may assume without loss of generality, that 
    	$A=B\subset U$ and $O= O_0$.  
    		Making $O$ smaller, if needed, we may assume that $O$ is a coordinate chart. Thus, on $O$
    	the notion of semiconcavity  are the same with respect to the metric structure and to the coordinate chart.
    	 
    	 The semiconcavity of $f$ has been verified in Corollary \ref{cor: petrun}.  Since
    	 on open subsets of $\R^n$ a function is $\mathcal C^{1,1}$ if and only if it is semiconcave and semiconvex, the properties (1) and (2) are equivalent.
    	 
    	 Clearly, (2) implies (3).
    	 
    	 For any $x\in O$, we find, by compactness, at least one shortest geodesic $\gamma _x :[0, f(x)]\to U$
    	 from $x$ to $A=B$.   Then $(f\circ \gamma _x)'=-1$ on $[0, f(x))$.  This shows, that for all $t\in (0, f(x))$, we have $|\nabla _{\gamma _x (t)}f|=1$.    If $f$ is $\mathcal C^1$
    	 it also implies $|\nabla _x f|=1$.   Hence, (3) implies (4).

    	 By the first formula of variation, $|\nabla_x f|=1$ if and only $x$ is connected with 
    	 $B$ by exactly one shortest geodesic $\gamma _x$. Note that for any such geodesic $f\circ \gamma _x (t) =f(x)- t$ for all $t\in [0,f(x)]$.    Moreover, for any 
    	 geodesic, $\tilde \gamma _x:[0,\epsilon) \to U$ satisfying the above equality for all 
    	 $t\in [0,\epsilon)$, the unique extension of $\tilde \gamma _x$ to length $f(x)$ ends on $B$.     This shows the equivalence of (4) and (5).

    	  It remains to show that (4)  implies (1).  In this case, gradient lines $\eta _x$ of $f$ in $O$
    	  are parametrized by arclength and $f\circ \eta _x$ has everywhere velocity $1$. This implies that any gradient line $\eta _x$ in $O$ is an (as always minimizing)  geodesic.
    	  
    	  Fix $x\in O$ and consider $\delta$ such that the ball of radius $3\delta$ around $x$ is contained in $O$.  Consider $C=C(K,\delta)$, the semiconcavity constant of distance functions in $O$ at distance $\geq \delta$, as in Subsection \ref{subsec: distspec}.
    	  We claim that $f$ is $-C$-convex on the ball $W$ of radius $\delta$ around $x$.
    	  
    	  Indeed, consider points $q_1,q_2\in W$ and their midpoint $m$.  Set $z=\eta_m (2\delta)$.  Then, for $i=1,2$,
    	  $$f(z)=f(m)+2\delta =f(m)+d(m,z)\leq f(q_i) +d(q_i,z) \;,$$
    	   due to the triangle inequality.  This implies
    	  $$f(m)\leq \frac 1 2 (f(q_1)+f(q_2))+ (\frac 1 2 (d(q_1,z) +d(q_2,z))- d(m,z)) \leq $$
    	  $$\leq    	  \frac 1 2 (f(q_1)+f(q_2)) +\frac  C 8  \cdot d^2 (q_1,q_2) \;.$$
    	  This implies the claimed semiconvexity of $f$.
    	\end{proof}

 \subsection{Cut locus}  As a combination of previous results we now  obtain 
 \begin{proof}[Proof of Proposition \ref{prop: cut}]
As before set  $f=d_A$.
 Due to Proposition \ref{prop: distance1}, the cut locus  $CL (A)$ of $A$   is exactly the closure
 $CL(A)= \bar X$  of 

 $$X=\{x\in M\setminus A \; : \; |\nabla _x f|<1\}\;.$$
 As we have seen in Corollary \ref{cor: petrun}, the set $X$ is invariant under the gradient flow of $f$. By continuity of the gradient flow, the closure $\bar C=CL(A)$ is invariant under the gradient flow as well.
 	\end{proof}

 As a  consequence we deduce the following  observation from \cite{Ghomi}:
 \begin{cor} \label{cor: ghomi}
 	Let $M$ be a complete Riemannian manifold. Let $A$ be a closed subset of $M$ and let $X=CL(A) \subset M\setminus A$ be the cut locus  of $A$.  Assume that the distance function $d_A$ is concave on $M$. 
 	
 	 Let $x\in M\setminus A$ be arbitrary and let 
 	$x_0\in CL(A)$ be  a point with $d(x,x_0)=d(x,CL(A))$.  Then $d_A(x_0) \geq d_A(x)$.
 \end{cor}

 \begin{proof}
 Assume the contrary, thus $d_A(x_0)<d_A(x)$. Consider a geodesic $\gamma$ from $x_0$ to $x$.
 Since $h:=d_A\circ \gamma$  is concave, the derivative of $h$ at $0$ is positive. In particular, $x_0$ is non-critical for the function $d_A$.  Thus, the gradient curve $\eta _{x_0}$
 of  $d_A$ starting at $x_0$
  is non-constant. Moreover, the angle between $\gamma'(0)$ and $\nabla _{x_0}  d_A $ is less than $\frac \pi 2$.
 
 By the first variation formula, the derivative of $l(t):=d(x,\eta _{x_0} (t))$ at $0$ is negative. Hence, for small $t>0$, we have 
 $$d(x, \eta _{x_0} (t)) < d(x, x_0)\;.$$
 Due to Proposition \ref{prop: cut}, $\eta _{x_0} (t)$ is  contained  $CL(A)$. Thus, $x_0$ is not a nearest point to $x$ in $CL(A)$. This contradiction  finishes the proof.
 	\end{proof}
 
 \subsection{Simple counterexample} The following example shows that without a completeness assumpion no higher smoothness of the distance function can be expected:

 \begin{ex} \label{ex: nonsmooth}
 	  Let $M$ be the flat non-complete manifold $\R^2\setminus C$, where $C$ is   the ray $C = \{(0,t)\; : \; t\geq 0 \}$.  Let $A$ be the singleton $(1,0) \in M$. It is not difficult to see (for instance, using that the completion of $M$ is CAT(0), \cite[Proposition 12.1]{LW-iso},  thus uniquely geodesic) that $Reg^{d_A}$ is the whole complement $M\setminus A$.
 	While $M$ is smooth and $A$ is a smooth submanifold, the distance function $f$ to $A$ is not smooth:
 	level sets of $f^{-1} (s)$ of $f$ are concatenations of parts of  Euclidean circles of radii 
 	$s$ and $s-1$ and thus not $\mathcal C^2$.  
 \end{ex}

  \section{Subsets of positive reach} \label{sec: reach}
 \subsection{Characterization} We are going to provide:
 \begin{proof}[Proof of Proposition \ref{prop: reach}]
   Proposition \ref{prop: distance} shows the equivalence of (1) and (2). Moreover, by 
Proposition \ref{prop: distance}, (3) implies (2).

It remains to show that (2) implies (3).  Thus, let $O$ be a neighborhood of $A$ such that
$f=d_A$ is $\mathcal C^{1,1}$ in $O\setminus A$.  Then $f$ is semiconvex on $O\setminus A$ and we need to show that $f$ is semiconvex around any point $x\in A$.  Consider a special ball
$U_x$ around $x$ of radius $r$.    Let $W$ be the ball of radius $\frac r 3$ around $x$.
For any $q_1,q_2 \in W$  with midpoint $m$ we either have $m\in A$ and then 
$$f(m)=0\leq \frac 1 2 (f(q_2)+f(q_2)) \;.$$
Or $m\notin A$. Then the gradient curve $\eta _m$ of $f$ starting in $m$ is a geodesic on $[0,\frac r 3]$. As  in the last part of the proof of Proposition \ref{prop: distance}, we deduce 
$$f(m) \leq   \frac 1 2 (f(q_1)+f(q_2)) +\frac  C 8  \cdot d^2 (q_1,q_2) \;,$$
for some $C$ depending only on $U_x$.  Hence $f$ is semiconvex in $W$, finishing the proof.
\end{proof}

\subsection{Geodesics in normal directions}
Now we derive:
\begin{proof}[Proof of Proposition \ref{lem-norm-geod}] 
	The distance function $f=d_A$ is semiconvex  in a neighborhood $O$ of $A$.  Thus, we have a well-defined (directional) differential $D_xf:T_xM\to \R$ and this differential coincides with the  distance to  the tangent cone $T_x A$, cf. \cite{Lyt-open}.
	
	By the definition of normal directions, we have $D_xf (h)=1$. In other word, $(f\circ \gamma ^h )')(0)=1$.   By semiconvexity of $f\circ \gamma ^h$, we have 
	$$\lim _{t\to 0} (f\circ \gamma ^h)'(t)=1\,.$$ 
	 By the first variation formula, this implies that for $t\to 0$, the angle 
	between $(\gamma^h)'(t)$ and the unique shortest geodesic from $\gamma ^h(t)$ to $A$  converges to $\pi$.  Thus the angle between the gradient curve $\eta^t$ of $f$ starting at $\gamma ^h (t)$ and $\gamma ^h$ at the point $\gamma ^h (t)$ converges to $0$.  Since, the gradient curves $\eta ^t $ are geodesics on a fixed interval $[0, s_0]$, we deduce that $\eta ^t :[0,s_0]\to O$ converge to $\gamma ^h$.
	
	Thus, $d(\gamma ^h (t), A)=t$ for all $t\in [0,s_0]$.  This implies the claim.
\end{proof}

\section{Harmonic coordinates and submanifolds} \label{sec: last}

\subsection{Harmonic  and distance coordinates} We provide:
\begin{proof}[Proof of Proposition \ref{prop: harm}].
	Thus, let $M$ be a manifold with two sided  curvature bounds.  Let $G:V\to \R^n$ be some distance coordinates on an open subset $V$ of $M$ and let $f:V\to \R$ be a harmonic function.
	
	We identify $V$ with $G(V)\subset \R^n$.  The   Riemannian metric $g$ of $M$ restricted to  is Lipschitz continuous on $V$, \cite[Theorem 13.2]{Ber-Nik}.

	Let $x\in V$ be arbitrary and consider  arbitrary harmonic coordinates $F:O\to \R^n$ defined on a neighborhood $O$ of $x$. Thus, the coordinates of $F$ are harmonic functions in $O$.
	
	Due to Nikolaev's theorem,  the pull-back Riemannian metric $ \tilde g:= (F^{-1}) {\ast}  (g)$ is $\mathcal C^{1,\alpha}$
	on $\tilde O:=F(O) \subset \R^n$, for all  $\alpha <1$, \cite[Theorem 14.2]{Ber-Nik}. In particular, $\tilde g$ is locally Lipschitz continuous.
	
	Then $F:(O,g)\to (\tilde O, \tilde g)$ is  an isometry between Riemannian manifolds with $\mathcal C^{0,1}$ Riemannian metrics.  And such an isometry 
	is always $\mathcal C^{1,1}$, \cite{Sabitov-smooth}.  See also \cite{LY} and \cite{T} for other proofs of this fact.
	Hence,  $F$ and $F^{-1}$ are of class $\mathcal C^{1,1}$.

	The  function $F^{-1} \circ f$ is harmonic on $\tilde O$. By elliptic regularity, \cite{Taylor}
	or \cite[p. 689]{Sabitov}, the function $F^{-1} \circ f$ is $\mathcal C^{3,\alpha}$.  In particular, it is $\mathcal C^{1,1}$. Hence, $f=F\circ F^{-1} \circ f$ is $\mathcal C^{1,1}$ on $O$ as well. 
	This finishes the proof.
\end{proof}

\subsection{Submanifolds} This subsection is devoted to 
\begin{proof}[Proof of Proposition \ref{prop: submanifold}]
Let $M$ be a manifold with two sided bounded curvature and let $N$ be a $\mathcal C^{1,1}$ submanifold with respect to distance coordinates.
	
	The statement is local. We fix  any point $x\in N$ and  consider a special neighborhood 
	$U_x$ of $x$ in $M$. Then we find a compact $\mathcal C^{1,1}$ submanifold $\hat N$ of $M$ which is contained in $ U_x$ and such that $N$ and $\hat N$ coincide in a neighborhood of $x$.

	Since the statement is local we may assume that $N=\hat N$.  We may further assume that $U_x$ is a distance chart and identify it with a ball in $\R^n$. Then the distance in $U$ is given by a Lipschitz continuous Riemannian metric $g$.  
	
	Due to \cite[Proposition 1.5]{Ly-conv}  and \cite[Theorem 1.2]{Ly-reach}, the space $N$ is $CAT(\kappa)$ for some $\kappa$.   It thus remains to show that  $N$ has curvature bounded from below.

	Applying Proposition \ref{prop: harm} and Nikolaev's approximation theorem, \cite[Theorem 15.1]{Ber-Nik}, we may further assume on $U$ there exist smooth Riemannian metrics $g^{\epsilon}$, which are uniformly Lipschitz continuous and converge to the Riemannian metric $g$. Moreover, the Riemannian manifold $(U,g^{\epsilon})$ have sectional curvature uniformly bounded from above and below.
	
	Smoothing the submanifold $N$ we find a family of smooth submanifolds $N^{\epsilon} \subset U$ with 
	uniform bounds on their $\mathcal C^{1,1}$-norms (thus any $N^{\epsilon}$ is a union of a uniform number of charts each of them  of bounded $\mathcal C^{1,1}$-norm), such that
	$N^{\epsilon}$ converge to $N$ in $\mathcal C^1$ sense.
	
	Considering $N^{\epsilon}$ with the the intrinsic metric induced by $g^{\epsilon}$, we see that $N^{\epsilon}$ converge to $N$ in the Gromov--Hausdroff metric (in fact, the convergence is much stronger).    Thus, it remains to show that all $N^{\epsilon}$ are Alexandrov space
	of curvature $\geq -\kappa$, for some fixed $\kappa$.  However, the $\mathcal C^{1,1}$ bounds of $N^{\epsilon}$ directly imply, that the second fundamental forms of $N^{\epsilon}$ are uniformly bounded (as subsets of the flat space $\R^n$ and therefore,  of the Riemannian manifold $(U,g^{\epsilon})$).  Applying the Gau\ss $ $ equation, we derive a uniform lower bound on the sectional curvatures of $N^{\epsilon}$. By Toponogov's theorem and compactness of $N^{\epsilon}$, the manifolds $N^{\epsilon}$ are Alexandrov spaces of curvature $\geq -\kappa$.   This finishes the proof. 
\end{proof}

\bibliographystyle{alpha}

\end{document}